\documentclass{amsart}
\usepackage{amsmath,amssymb}
\usepackage{amsthm,amscd,amssymb,appendix,verbatim,epsf,amsmath,eurosym,amsfonts,mathrsfs,graphicx}
\usepackage{tikz-cd}
\usepackage[colorlinks=true,linkcolor=blue,citecolor=blue]{hyperref}
\usepackage{amsmath, amsthm, amscd, amsfonts, amssymb, graphicx, color}
\usepackage{tikz-cd}
\usepackage{parskip}
\newcommand{\be}{\begin{equation}}
\newcommand{\ee}{\end{equation}}
\newcommand{\bea}{\begin{eqnarray}}
\newcommand{\eea}{\end{eqnarray}}
\newcommand{\bean}{\begin{eqnarray*}}
\newcommand{\eean}{\end{eqnarray*}}          
\usepackage[english]{babel}
 \newtheorem{theorem}{Theorem}
 
\newtheorem{proposition}[theorem]{Proposition}
\newtheorem{definition}[theorem]{Definition}

\begin{document}
\setlength{\parindent}{15pt} 
\setlength{\parskip}{0.2mm}

\title[The Bishop family of holomorphic discs]{The Bishop family of holomorphic discs: regularity and higher index}
\author{Brendan Guilfoyle}
\address{Brendan Guilfoyle\\
          Department of Technology, Engineering and Mathematics\\
          Faulty of Science and Informatics \\
          Munster Technological University\\
          Tralee\\
          Co. Kerry\\
          Ireland.}
\email{brendan.guilfoyle@mtu.ie}
\author{Wilhelm Klingenberg}
\address{Wilhelm Klingenberg\\
 Department of Mathematical Sciences\\
 University of Durham\\
 Durham DH1 3LE\\
 United Kingdom.}
\email{wilhelm.klingenberg@durham.ac.uk }

\begin{abstract}
We prove $C^{k/2,\alpha/2}$ -regularity  up to a non-umbilic elliptic complex point for the Bishop family of holomorphic discs with boundary in a $C^{k,\alpha}$  regular real surface. Furthermore, we prove existence and regularity  of holomorphic discs near certain complex points of index $\ge 2$. 

The proof employs a novel blow-up of the real surface which resolves the complex point to a pair of totally real surfaces and leads to a $\mathbb{Z}_2$ -equivariant  Riemann-Hilbert problem for holomorphic annuli. The index is computed to be 1 and the problem is shown to be Fredholm-regular.
\end{abstract} 

\date{\today}
\maketitle
\setcounter{section}{0} 

\section{Introduction}
\noindent
We study the existence and regularity of the solution space of a Riemann-Hilbert problem for  holomorphic curves of disc- and annulus-type in a complex surface $(M,J)$ with the boundaries constrained to lie in a fixed embedded real surface
\be\label{e:1}
  F:  S \quad \hookrightarrow \quad M.
\ee
Thus $M$ is a real 4-manifold endowed with a complex structure $J$, and $F(S) \subset M$ is a real surface, also denoted $S$, and we seek maps $f:(D,\partial D)\rightarrow(M,S)$ such that $\bar{\partial}_Jf=0$. 

A point $ q\in S$, is called a \textit{complex point of S} if the real plane $T_ q  S\subset(T_{q} M,J_q))) \cong {\mathbb{C}}^2$ is a complex line. This may be characterized by the vanishing of a section of a certain bundle over $S$ (see equation (\ref{e:4})) and a complex point is called \textit{elliptic} if this section is transverse to the zero section, with intersection number one. In this case, the point  is called \textit{umbilic (non-umbilic) elliptic} if the eigenvalues of the linearization of the section at $q$ do  (do not) coincide. 

To formalize the moduli space of holomorphic discs, let $D\subseteq{\mathbb{C}}$ be the unit disc, Aut(D)  be  its M\"obius group, $p > 2$,  and 
\[
{\mathcal D}(M, S):=\{f:(D,\partial D)\to(M, S)\:|\:f\in W^{1,p}(\overline D) {\mbox{ and }} \bar\partial f=0\}/\mbox{Aut}(D).
\]
Thus, ${\mathcal D}(M, S)$ is the set of unparameterized holomorphic discs of Sobolev class $W^{1,p}(\overline D)$ in $M$ with boundary in $S$. Now 
\[
{\mathcal D}(M, S) \subset W^{1,p}(D,\partial D; M,  S)/\rm{Diffeo}(D),
\]
where $\rm{Diffeo}(D)$ is the diffeomorphism group of the disc. The right hand side is the set of unparameterized discs with boundary of class $W^{1,p}$. This is a Frechet manifold. A point $ q\in S$ may be regarded as a (degenerate) element of ${\mathcal D}(M, S)$ - the constant (thereby holomorphic) map to $q$. Let ${\mathcal D}_ q(M, S)$ be a neighbourhood of $ q\in{\mathcal D}(M, S)$ in the induced topology. 

Our main result is the following.
\\
\begin{theorem} \label{t:1} 
Let $F: S\;{\hookrightarrow}\; M$ be a real surface in a complex surface and let $ q\in S$ be a non-umbilic elliptic point. Assume that $F$ is $C^{k,\alpha}$-smooth, for some $k \ge 1, \alpha > 0$. Then  
\begin{enumerate}
\item[(a)]  ${\mathcal D}_ q(M, S)$ is a $C^{\frac{k}{2},\frac{\alpha}{2}}$ -smooth 1-dimensional submanifold of the Frechet manifold $W^{1,p}(D,\partial D;M, S) ${\rm /Diffeo(D)} with boundary point $ q\in\partial{\mathcal D}_ q(M, S)$,
\item[(b)] In a neighborhood of $ q\in\partial{\mathcal D}_ q(M, S)$, the elements of ${\mathcal D}_ q(M, S)$ have uniformly bounded $C^{\frac{k}{2},\frac{\alpha}{2}}(\overline D )$-norm, 
\item[(c)] $T_ q{\mathcal D}_ q(M, S)={\mathcal D}_0({\mathbb{C}}^2,  S_2)$, being a family of holomorphic discs given explicitly in Proposition \ref{p:4}.
\end{enumerate} 
\end{theorem}
Here $C^{k,\alpha}$ refers to H\"older spaces and claim (a) gives the H\"older regularity of the dependence of the holomorphic discs in ${\mathcal D}_ q(M, S)$ on a real parameter up to the boundary $q\in {\partial\mathcal D}_ q$ with reference to the Frechet-topology of Sobolev space $W^{1,p}(D,\partial D;M, S){\rm /Diffeo(D)}$.

By elliptic regularity for the Riemann-Hilbert boundary problem, the  points (i.e. holomorphic discs) of ${\mathcal D}_ q(M, S)$ with totally real boundary are in $C^{k,\alpha}(\overline D)$. Claim (b) asserts that their $C^{\frac{k}{2},\frac{\alpha}{2}}$-norm is bounded in a neighborhood of the complex point
$ q \in \partial{\mathcal D}_ q(M, S)$.

Claim (c) says that the 1-manifold with boundary ${\mathcal D}_ q(M, S)$ is modeled near the boundary point $q$ by ${\mathcal D}_0(T_m M,S_2)$. In particular, as a subset in $M$,  the family of discs ${\mathcal D}_ q(M, S)$ foliates an embedded three-manifold of regularity $C^{\frac{k}{2},\frac{\alpha}{2}}$ up to $S$. Our proof also shows that if $F$ is real analytic, then so is  ${\mathcal D}_ q(M, S)$ in claim (a) and (b).

Our main result implies the following.
\\
\begin{theorem}\label{t:2}
Let $F:  S\hookrightarrow (M,J)$ be an embedded real surface in a complex surface and $ q\in  S$ be a non-umbilic elliptic point. Then, for a neighborhood $S_ q$ of $q$ in $S$, there exists an embedded real hypersurface with boundary $ (N,   S_ q , q)\hookrightarrow (M,  S , q)$ such that
 \begin{enumerate}
\item[(a)] $N$ is Levi-flat and foliated by $J$-holomorphic discs in $(M,J)$ and boundary in $  S$,
\item[(b)] if $S$ is $C^{k,\alpha}$-smooth, then $N$ is $C^{\frac{k}{2}, \frac{\alpha}{2}}$-smooth up to the boundary $S$ and if 
$S$ is  $ C^{\omega}$,  then so is $N$,
\item[(c)] the family of discs converges to the constant map at $q$.
\end{enumerate}
\end{theorem}

The method of proof we employ extends to complex points of higher index, as long as the lowest order of their Taylor expansion has a certain quadratic factor. That is, both Theorems can be improved by assuming that the complex point is of any index $>0$ of the form given in equation (\ref{e:Shigher}).  

As this is the first such result about higher index complex points we state the equivalent of Theorem \ref{t:2} for higher index in full. 
\\
\begin{theorem}\label{t:3}
Let $F:S\hookrightarrow (M,J)$ be an embedded real surface in a complex surface and $ q\in S$ a complex point with Taylor series of the form given in equation (\ref{e:Shigher}) for some $n \ge 3.$  Then, for a neighborhood $S_ q$ of $q$ in $S$, there exists an embedded real hypersurface with boundary $ (N,S_q,q)\hookrightarrow (M,S,q)$ such that
\begin{enumerate}
\item[(a)] $N$ is Levi-flat and foliated by $J$-holomorphic curves of disc-type in $(M,J)$ and boundary in $  S$,
\item[(b)] Given $n$ as in equation \eqref{e:Shigher}, assume that $k\ge n$. If $S$ is $C^{k,\alpha}$ ($C^\omega$) -smooth, then $N$ is 
$C^{(k-n + 2)/2,\alpha/ 2}$  ($C^\omega$) -smooth  up to the boundary 
$S$,
\item[(c)] $N$ is tangent at $q$ to the hypersurface swept foliated by the family of discs in equation \eqref{e:fa}.
\end{enumerate}
\end{theorem}

In \cite{B}, Bishop  proved existence and continuity of the family $\mathcal{D}_q $ up to the elliptic complex point $q$. This was improved by Kenig and Webster in \cite{KW}   who established $C^{(k-7)/3}$-regularity for $N$. The cases of a real analytic and $C^\infty$ regularity of $S$ were shown in \cite{Hu} to incur no loss of regularity of $N$.

Huang provides an example of a real surface $S$ with an {\it umbilic elliptic}  point $q$ for which the regularity of $N$ as Theorem \ref{t:2} (a) is no higher than $C^{(k+1)/2}$ \cite{Hu}. While our proof does not apply for such umbilic elliptic complex points, we conjecture that Theorem \ref{t:3} is optimal at any elliptic point. For related work on CR-singularities see \cite{FH18} \cite{HY16} \cite{HY17} \cite{W95}.

In their paper, the authors of \cite{KW} apply a Picard-type iteration scheme to pass from the osculating quadric $S_2$, which bounds the family of discs \eqref{e:fa} below, to the given surface $S$. In contrast, our proof is geometric and employs an ambient blow-up that regularizes the CR-singularity. This gives rise to a family of holomorphic curves of annulus-type in a branched cover  of $(M,J)$ whose boundary lie in a resulting totally real surface $\Lambda_0^{\pm}$, which is the union of two embedded M\"obius bands. These curves converge to an annulus in the exceptional fibre of the blow-up ${\rm Bl}_0(\mathbb{C}^2)$.

The CR-singularities, i.e. complex points, that have been considered previously have been elliptic, where the Keller-Maslov index is 1. In general, higher index complex points are associated with larger moduli spaces of holomorphic curves, but their construction is difficult and has not been undertaken previously. Complex points on a real surface in the complex surface $T{\mathbb S}^2$ with its canonical complex structure have particular significance, as $T{\mathbb S}^2$ can be identified with the space of oriented affine lines in ${\mathbb R}^3$. 

This identification is geometric and the associated neutral K\"ahler structure on $T{\mathbb S}^2$ carries much of the underlying Euclidean geometry. Thus the set of oriented normal lines to a surface $S'\subset{\mathbb R}^3$ form a surface $S$ in $T{\mathbb S}^2$ that is totally real on $S$ away from umbilic points on $S'$. The umbilic points of $S'$ give complex points on $S$ and the index determines the local behaviour of the principle foliation of the surface $S'$ \cite{GK05}.

In this setting, index 2 complex points of lowest order correspond to the umbilic point at the north or south pole of a rotationally symmetric surface in ${\mathbb R}^3$.  A long-standing local index Conjecture of Loewner, which implies a global Conjecture of Carath\'eodory, states that 2 is the highest complex index than can arise in this manner from an umbilic point on a surface in ${\mathbb R}^3$ \cite{GK19}\cite{GK20}\cite{GK24a}\cite{GK25}. The Conjecture of Loewner remains open, although the weaker local index bound of 4 has been proven in \cite{GK24c}.

Furthermore, we point out that the problem of finding a Levi-flat hypersurface with prescribed boundary is analogous to the classical Plateau problem for minimal surfaces in Euclidean three-space. In  the Plateau problem for minimal surfaces in Euclidean space, Hildebrandt \cite{Hi} proved that the surface retains the same regularity up to the boundary as that of the boundary condition.

Filling by Levi-flat hypersurfaces in $T{\mathbb S}^2$ can be quite rigid - as can be seen for minimal and CMC surfaces in \cite{MS23}, where the holomorphic discs correspond to parts of round spheres which have higher order contact with a minimal surface. For recent results on constructing Levi-flat hypersurfaces in $T{\mathbb S}^2$ around complex points of higher index using parabolic evolution equations, see \cite{GK24b}.  

The next section contains the required background on the blow-up of a complex point and the resulting boundary problem. In Section \ref{s:3} we prove Theorems \ref{t:1} and \ref{t:2} by computing the index of the associated equivariant boundary value problem and this is extended to a proof of Theorem \ref{t:3} in Section \ref{s:4}.

\vspace{0.1in}
\section{The quadratic osculant at an elliptic point}\label{s:2}

Let $(S,F,M)$ be as in map (\ref{e:1}) and $ q \in  S$ an isolated complex point of $F$. Choose a conformal structure $j$ on $S$ near $q$ by setting $j|_q:=F^*|_q J$ and extend $j$ to a neighborhood of $ q \in  S$. Let  $\Omega (S)={\rm Hom_{\mathbb{R}}}(T S, {\mathbb{C}})=\Omega^{1 0}(S) \oplus \Omega^{0 1}(S)$, where $\Omega^{1 0}(S)$ and $\Omega^{0 1}(S)$ is the bundle of holomorphic and anti-holomorphic 1-forms over $(S,j)$, respectively.

Now define 
\bea 
\partial F &:=& {\textstyle{\frac{1}{2}}}(dF-JdF\circ j)\in\Omega^{10}(F^*T^{10}M)\cong \Omega^{10}(S)\otimes F^*T^{10}M \\ \label{ba}
\bar{\partial}F &:=&{\textstyle{\frac{1}{2}}}(dF+JdF\circ j) \in\Omega^{01}(F^*T^{10}M)\cong \Omega^{01}(S)\otimes F^*T^{10}M. 
\eea
It follows that $\bar{\partial}F(q)=0$. Then
\be \label{e:4}
\partial F\wedge\bar{\partial} F\in \Omega^{10}(F^*T^{10}M)\wedge 
\Omega^{01}(F^*T^{10}M) \cong
{\rm det_{\mathbb{R}}} T^*  S \otimes {\rm det}_{\mathbb{C}} F^*T^{10}M.
\ee
This is  a section of a complex line bundle over $S$ which vanishes at the complex point $ q \in  S$ since the section \eqref{ba} vanishes at $q$. Since $F$ has maximal rank over ${\mathbb{R}}$, the section (\ref{e:4}) vanishes if and only if one (and indeed only one) of $\partial F, \;\bar\partial F$ vanishes.

 We choose holomorphic coordinates $(z,w)$ of $M$ near $q$ with $(z,w)(q)=(0,0), T_q  S=\{w=0\}$. Then for some $a,b,c \in {\mathbb{C}}$, and for a complex coordinate of  $\zeta$ with $\zeta(q) =0$, the embedding $F$ may be smoothly parametrized to the form 
\be \label{e:2} (z,w) \circ F(\zeta)=(\zeta, a \zeta^2 + 2b \zeta \bar\zeta + c \bar \zeta^2 + O_3(\zeta, \bar\zeta)).
\ee
Namely $F(0) = (0,0)$,
\[
\bar\partial F=(0, 2b\zeta + 2c\bar\zeta + O_2(\zeta,\bar\zeta))d \bar\zeta,
\]
and therefore
\[
\partial F \wedge \bar \partial F=(2b\zeta + 2c\bar\zeta + O_2(\zeta,\bar\zeta))d\zeta\wedge d\bar\zeta.
\]
 \\
\begin{definition}\cite{B}\label{d:4}
We say  the complex point $q$ is  {\em nondegenerate} if the section $\partial F \wedge \bar \partial F$ has non-vanishing rank over ${\mathbb{R}}$ at $\zeta =0$ This is the case if and only if one of $b, c$ does not vanish. 
In this case  $q$ is called {\em elliptic (parabolic, hyperbolic)} if  $|b|-|c| > 0,=0, < 0.$ Furthermore, $q$ is called {\em umbilic elliptic (non-umbilic elliptic)} if the eigenvalues of the linearization at $q$ of $\partial F \wedge \bar \partial F$ do (do not) coincide. This is the case if and only if $|c|=0$  ($|c|$$ \neq 0$). 
\end{definition} 

We now proceed to give a family of holomorphic discs with boundary near an elliptic point of a {\it quadratic} surface $S_2$. The following Proposition is easily verified and gives the model family  ${\mathcal D}_0({\mathbb{C}}^2,  S_2)$ for Theorem \ref{t:1}(c). 
\\
\begin{proposition}\label{p:4}
Let $(0,0)$ be an elliptic point of the embedded quadratic surface 
\begin{equation}\label{e:osc}
 S_2=\{ w=a z^2 + 2b z \bar z + c \bar{z}^2\} \subset {\mathbb{C}}^2_{z,w}, 
\end{equation}
namely assume that $|b|-|c| > 0.$ Consider the coordinates $\zeta, \omega$ as defined by: 
$$ (z,w)=\left(e^{i\delta}\zeta,\;\; \omega + (a e^{2i\delta}-{\textstyle{\frac{|c|}{|b|}}}b)\zeta^{2} \right),$$  
where $\delta=(1/2)(arg(b)-arg(c)). $ Then the family ${\mathcal D}_0({\mathbb{C}}^2,  S_2)$ parametrized by $t>0$ of discs given by 
\begin{equation}\label{e:fa}
    \{ \left({\textstyle{\frac{|c|}{|b|}}} \zeta^{2} + 2 \zeta \bar \zeta +{\textstyle{\frac{|c|}{|b|} }}\bar \zeta^{2} < t,\;\; \omega=tb\right)\}_{t>0} \subset {\mathbb{C}}^2_{\zeta ,\omega},
\end{equation}
are disjoint holomorphic discs with boundary in  $S_2$ shrinking to the elliptic complex point $(0,0) \in  S_2$ as $t\to 0$.
\end{proposition}

We now define and study a holomorphic double cover of ${\mathbb{C}}^2$ which resolves a non-umbilic elliptic point $ q $ of a real surface $F:S \hookrightarrow M$ near a non-umbilic umbilic $q$.
\\
\begin{definition}
Let $(S,F,M)$  be as in equation \eqref{e:1} and $ q \in S$ be a non-umbilic elliptic point, and $(z,w)$ be holomorphic coordinates based at $m=F(q) \in M$ as in equation \eqref{e:2}. Recall from equation \eqref{e:osc}  of the quadratic osculant $S_2 \subset \mathbf{C}^2_{z,w}$ of $S$ at $q$. Then define
\be \label{e:co}
F_2:{\mathbb{C}}^2_{z_1,w_1} \to  {\mathbb{C}}^2_{z,w} , \; F_2(z_1,w_1)=(z_1, az_1^2 + 2bz_1w_1 + c w_1^2 ).
\ee
\end{definition}

\begin{proposition} \label{p:6}
Let $q$ be a nondegenerate complex point of $S$, $F$ as in equation \eqref{e:2} and  $S_2$ be as in equation \eqref{e:osc} and $F_2$ as above. Assume in addition that $q$ is non-umbilic elliptic, namely that $ c \neq 0 $.  Then
\begin{enumerate}
\item[(a)] the map $F_2$ is a holomorphic double branched covering with linear deck transformation:
\[
\tau:{\mathbb{C}}^2_{z_1,w_1}\to {\mathbb{C}}^2_{z_1,w_1},
\;\;
\;\;F_2 \circ \tau =F_2,
\]
\item[(b)] $F^{-1}_2 ( S_2)=L^+_ q \cup L^-_ q $
is the union of two transversal totally real planes $ L^\pm_ q \subset
{\mathbb{C}}_{z_1,w_1}$,
\item[(c)] The point $ q \in  S$ is   elliptic (parabolic, hyperbolic) if and only if there is no (one, two) complex line(s) in ${\mathbb{C}}^2_{z_1,w_1}$ which intersect both planes $L^\pm_ q$ in real lines. 
\end{enumerate}
\end{proposition}

\begin{proof}
    Part (a) follows from the assumption $c \neq 0$. To prove part (b), note that the map $\tau$  is obtained by solving $F_2 \circ \tau (z_1, w_1)=F_2(z_1,w_1)$ for $\tau$, which gives a linear map represented by the matrix
\be\label{matrixtau}
\tau=\bigg(\begin{array}{cc} 1 & 0 \\ -b/c & -1\end{array}\bigg).
\ee
It has eigenvalues $\pm 1$ with eigenvectors $(1 , -b/(2c) )$ and  $(0, 1)$. Next note that the set $F_2^{-1}( S_2)\subset {\mathbb{C}}^2_{z_1,w_1}$ contains the plane $L_ q^+:= \{w_1=\bar{z_1}\},$ which is totally real. Since $F_2^{-1}( S_2)$ is invariant under $\tau$, we see that $L_ q^-:= \tau \{ w_1=\bar{z_1} \}=\{w_1 =-\frac{b}{c}z_1-\bar{z_1}  \}$ is also in $F_2^{-1}( S_2)$, and is totally real by inspection and indeed since $\tau$ is holomorphic. By inspection, the two totally real planes $L_ q^{\pm}$ are mutually transversal, which completes the proof of part (b). 

Next note that the complex lines in ${\mathbb{C}}^2_{z_1,w_1}$ which intersect $L_ q^+$ in a real line may be represented by ${\mathbb{C}} \cdot (e^{it}, e^{-it}) ; \;t \in {\mathbb{R}} $, and those not transversal to $L_ q^-=\{ (- b/c ) z_1-w_1=\bar{z_1}\}$ are ${\mathbb{C}} \cdot (e^{is},(- b/c) e^{is}-e^{-is} ) ; s \in {\mathbb{R}} $. The intersection of these sets of complex lines corresponds to the condition $e^{-2it} =-b/c -e^{-2is}$. This equation has no solution (one, two solutions) iff $|b /c|  > 2 (\;=, \; <\; )$. This completes the proof of part (c). 

\end{proof} 

Let $\rm{Bl}_0(\mathbb{C}^2)$ be the blowup  of $\mathbb{C}^2$ at the origin. We denote by  $\rm{Bl}_0^{-1} : \rm{Bl}_0(\mathbb{C}^2) \to \mathbb{C}^2$ the associated blow-down map.
Now we lift a neighborhood of a non-degenerate and non umbilic complex point $ q \in  S$ through $F_2$ to a pair of  real surfaces.  
The proper transform of $L_q^\pm$ is denoted by 
$\Lambda_0^\pm$ and the result of applying $F_2^{-1}$ and blowup is as follows.
\vspace{0.1in}
\begin{center} 
\begin{equation}  \label{diag}
    \begin{tikzcd} 
 S \cap B_\epsilon  &    \Lambda^+ \cup \Lambda^-\arrow[l,"F_2"]  &\Lambda_0^+ \cup \Lambda_0^- \approx  \Lambda^+\# {\mathbb{R}}P^2 \cup \Lambda^-\# {\mathbb{R}}P^2 \arrow[l, "\rm{Bl}_0^{-1}"] \\[-15pt]
\cap& \cap &\cap \\[-15pt]
{\mathbb{C}}^2_{z,w}&{\mathbb{C}}^2_{z_1,w_1}\arrow[l,"F_2"]&{\rm{Bl}}_0({\mathbb{C}}^2) \approx {\mathbb{C}}^2_{z_1,w_1} \# {\mathbb{C}}P^2\arrow[l, "\rm{Bl}_0^{-1}"]
\end{tikzcd}
\end{equation}
\end{center}
Here $B_\epsilon$ is the ball of radius $\epsilon$ centered at the origin in ${\mathbb C}^2_{z,w}$ and containing a coordinate neighborhood of $q \in S \cap B_\epsilon$.

\begin{proposition} \label{p:blowup}
Assume that $ q\in S$ is a nondegenerate complex point which is non-umbilic elliptic. Then there exists  $\epsilon > 0$ with
\begin{enumerate}
\item[(a)]  $F_2^{-1}( S \cap B_\epsilon) $ is a union of two embedded totally real discs which we denote by $\Lambda^+ \cup \Lambda^- \subset {\mathbb{C}}^2_{z_1,w_1}$ and which intersect transversely at the origin and $T_{(0,0)}\Lambda^\pm=L^\pm_{ q}$,
\item[(b)] If $S\in C^{k,\alpha}$ $($ $C^\omega$$)$ then $\Lambda^\pm\in C^{\frac{k}{2},\frac{\alpha}{2}}$ $($ $C^\omega$$)$,
\item[(c)] Let $\rm{Bl}_0^{-1} :\rm{Bl}_0({\mathbb{C}}^2)\to {\mathbb{C}}^2_{z_1,w_1}$ be the blow-down map, then the proper transform of $L^\pm_q \subset{\mathbb{C}}^2_{z_1,w_1}$ is  $\Lambda^\pm \hookrightarrow \rm{Bl}_0({\mathbb{C}}^2)$ of  are embedded, disjoint, and totally real M\"obius bands,
\item[(d)] The holomorphic deck transformation $\tau$ of ${\mathbb{C}}^2_{z_1,w_1}$ extends to a holomorphic transformation which we also denote by $\tau : \rm{Bl}_0({\mathbb{C}}^2) \to \rm{Bl}_0({\mathbb{C}}^2).$ The fixed points of $\tau$ are on the exceptional fibre and correspond to the eigenspaces of the matrix  $(\ref{matrixtau})$. 
\end{enumerate}
\end{proposition}

\begin{proof}
    Claim (a) comes down to the equation
\[
az_1^2 + 2bz_1w_1 + c w_1^2=az_1^2 + 2bz_1\bar{z}_1 + c \bar{z}_1^2 + O_3(z_1,\bar{z}_1).
\]
It follows from Proposition \ref{p:6} (b) and the smooth Weierstrass Preparation Theorem \cite{GG} that there are two solutions of the form  $w_1=w_1(z_1,\bar{z}_1)$ and tangent at the origin to $\Lambda^\pm$. 
The regularity claimed in (b) also follows from the Preparation Theorem, since the problem involves a quadratic branching of the covering map $F_2$. 

Claim (c) follows from part (c) of Proposition \ref{p:6}. Namely the assumption of $q$ being elliptic implies that the proper transform, of the pair intersecting pair $\Lambda^\pm$ is {\it embedded} and results in glueing two copies of the real projective plane into $\Lambda^\pm$. This gives two M\"obius bands. 

Claim (d) follows from the fact that a ${\mathbb C}$-linear map takes lines to lines, therefore $\tau$ extends to the exceptional fibre, mapping it to itself.
\end{proof}

\section{An equivariant Riemann-Hilbert boundary problem}\label{s:3}
\noindent
Assume now that $ q\in S$ is a non-umbilic elliptic complex point and let $\mathbb{C}P^1 \hookrightarrow {\rm Bl}_0(\mathbb{C}^2)$ be the exceptional fibre. Then by Proposition \ref{p:blowup},  $\mathbb{C}P^1 \cap {\rm Bl}_0^{-1}(\Lambda_0^\pm)$ are two disjoint curves. These loops bound a holomorphic  annulus $A$ :
\be
(A,\partial^\pm A)\hookrightarrow ( {\rm Bl}_0(\mathbb{C}^2), {\rm Bl}^{-1}_0(\Lambda_0^\pm)).
\ee
Associated to this holomorphic curve with boundary in a totally real surface is the normal Maslov index $\mu$ as defined in \cite{MS12}.
\\
\begin{proposition} \label{tau}
In the above situation, we have the following:
\begin{enumerate}
\item[(a)] The normal Maslov index $\mu(A, {\rm Bl}_0^{-1}(\Lambda^\pm))$ vanishes,
\item[(b)] $\tau(A,\partial^\pm A)=(A, \partial^\mp A)$, 
\item[(c)] The quotient $(A,\partial^\pm A)\to (A,\partial^\pm A) {\rm mod}{\{1,\tau}\}$ is diffeomorphic to a disc, where  the quotient map has two branching points.
\end{enumerate}
\end{proposition} 
\begin{proof}
    The orientation of the two boundary components of $A$ is opposite and modulo sign they contribute the same number to the total Maslov index, thus adding up to zero. This proves part (a).

Part (b) follows from $\partial^\pm  A  = {\rm Bl}_0(\mathbb{C}^2) \cap \pi_0^{-1}\Lambda^\pm$, by definition of $A$, and $\tau \Lambda^+=\Lambda^-$ by Proposition \ref{p:6} (b), which implies $\tau (\partial^- A)=\partial^+ A$. Now since $\tau$ preserves connectivity, we have $\tau A=A$. This proves (b). 

For part (c) note that $(A, \partial^\pm A)$ admits a $\mathbb{Z}^2$-action given by $\{ {\tau}, {\tau}^2 \equiv \rm{id}  \}$. The eigenspaces of $d \tau (0)$ as given in equation \eqref{matrixtau} are spanned by the vectors $(1, -b/c)$ and $(0,  1)$. In homogeneous coordinates $[z_1:w_1]$, these correspond to the fixed points $-b/(2c), \infty \in \mathbb{C}P^1 $ of $\tau$. Moreover, the map $\tau|_{\mathbb{C}P^1}$ takes the form  $\tau([z_1:w_1])=-b/c-[z_1:w_1]$. 
By the proof of Proposition \ref{p:6}, the curves $\partial A \subset {\rm Bl}_0(\mathbb{C}^2) $ are parametrized by $[z_1:w_1]=e^{-2it}$,  $[z_1:w_1]=-b/c-e^{-2is}$ in homogeneous coordinates, which describe circles of radius one around $[z_1:w_1]=0$ and $[z_1:w_1]=-b/c$. Since $|b/(2c)| > 1$ by the elliptic nature of the underlying complex point, it follows that $-b/(2c)$ lies outside of both circles and therefore in the annulus $A \subset \mathbb{C}P^1$. Clearly the other fixed point $\infty$ also does and it follows that the quotient map $(A, \partial A) \to (A, \partial A)/\{ \rm{id}, \tau\}$ is a double branched covering of the annulus with two distinct branching points, and $(A, \partial A)/\{ \rm{id}, \tau\}$ is diffeomorphic to a disc.
\end{proof}
\noindent
 We now restrict the tautological bundle $p : {\rm Bl}_0(\mathbb{C}^2) \to \mathbb{C}P^1$ to $A \subset \mathbb{C}P^1$ and obtain a bundle with boundary  $(p^{-1}A, p^{-1}\partial A) \to (A, \partial A)$. Apply a diffeomorphism of ${\rm Bl}_0(\mathbb{C}^2)$ fixing the base $\mathbb{C}P^1$ of the bundle $p$ so that the real surfaces  $ \pi_{0}^{-1}\Lambda^{\pm} \subset p^{-1}\partial A$ lie over the boundary of the base $\partial A$ and so that they fibre with one real dimensional fibres.  Denote the surfaces $ \pi_{0}^{-1}\Lambda^{\pm}$   by $\Lambda_0^\pm$. 
 
 We arrive at a bundle
\be \label{b}
(p^{-1} A , \Lambda_0^\pm) \stackrel{p}\longrightarrow (A, \partial A).
\ee
This is a real plane bundle over $\overline{A}$ with a real line subbundle $\Lambda_0^\pm$ over the boundary $\partial A$.
A hermitian connection $\nabla$ on the total space induces a decomposition 
\be \label{cs}
 T (p^{-1} A) =TA \oplus VA
\ee   
into horizontal and vertical sub-bundles, where the left hand side stands for the tangent bundle of the total space of (\ref{b}) and $VA$ stands for the fibres of (\ref{b}). Along the base of the bundle (\ref{b}), the complex structure $J$ has the form 
\be \label{J0}
 J|_A \times  0=\bigg(\begin{array}{cc} j & 0\\ 0 &
i\end{array}\bigg).
\ee
Here, $j$ denotes the complex structure on $A$, and $i$ that on the fibre of the tautological bundle at the null dection $A \times  0$. The graph of a section $\varphi\in\Gamma(p^{-1}A,\Lambda_0^\pm)$  of the bundle \eqref{b} is a real surface in the total space whose boundary lies in  $\Lambda_0^\pm$:
\be
{\rm graph} \;\varphi=\{(z,\varphi(z));z \in A\} \subset p^{-1}A,\;\;
\partial\;({\rm graph}\;\varphi) \subset \Lambda_0^\pm. 
\ee
This surface is $J$-holomorphic if its tangent bundle
\[
\mbox{graph}\nabla\varphi= \{(v,\nabla\varphi\cdot v); v\in TA\},
\] 
is $J$-invariant:
\[
\mbox{graph}\nabla\varphi=J\;\mbox{graph}\nabla\varphi
\equiv \{ ((J_{11} + J_{12}\nabla \varphi)v, 
(J_{21} + J_{22}\nabla \varphi)v) ; v\in TA \},
\]
where $J_{ij}$ denote the entries of the matrix of $J \in {\rm Hom}_{\mathbb{R}}(Tp^{-1}A,Tp^{-1}A  )$ with respect to the the decomposition (\ref{cs}), i.e. $J_{11}\in {\rm Hom}_{\mathbb{R}}(TA, TA)$, $J_{12}\in {\rm Hom}_{\mathbb{R}}(VA, TA)$ etc. We now define the differential operator
\be \label{cr}
\bar{\partial}:\Gamma(p^{-1}A, \Lambda_0^{\pm} )\to\Omega^{01}(p^{-1}A),
\;\;\;\bar{\partial}\varphi=\nabla\varphi- K\nabla \varphi  ,
\ee
where  $K\nabla \varphi=(J_{21} + J_{22}\nabla \varphi)\circ (J_{11} + J_{12}\nabla \varphi)^{-1}$. The graphs in $(p^{-1}A,J) $ of solutions of $\bar{\partial} \varphi=0$ are $J$-holomorphic curves on which $J$ induces the orientation that is graphically compatible with that of $(A,j)$. We denote by $\tau$ the $J$-holomorphic involution of the relative bundle \eqref{b} as induced in the preceeding chapter. We have
\[ \tau : (p^{-1}A,\Lambda_0^{\pm}, J) \to (p^{-1}A,\Lambda_0^{\pm}, J).
\]
We are interested in $\tau$-invariant $J$-holomorphic curves of annulus-type $[A \times 0]$ that are sections of $(p^{-1}A,J)$ with boundary values in $\Lambda_0^{\pm}$. That is we seek solutions of
\be \label{cr0}
\bar{\partial}\;\varphi=0,  \;\;\;
\varphi \in \Gamma_\tau(p^{-1}A, \Lambda_0^{\pm}),
\ee
where $\Gamma_\tau$ denotes the $\tau$-invariant sections of \eqref{b}. To this end we linearize the problem as follows. Denote by 
\[
R:= (p^{-1}A|_{\partial A})\cap T\Lambda_0^{\pm}\hookrightarrow p^{-1}A|_{\partial A},
\]
the linearization of the boundary condition, which is a real line sub-bundle of $p^{-1}A|_{\partial A}$. Note that the linearization of $\tau$ along the zero section 
$ O_{p^{-1}A} = A \times 0$ acts on the pair 
\be\label{bdry}
d\tau(O_{p^{-1}A}): (p^{-1}A,R) \to (p^{-1}A,R).
\ee 

The linearization of  $\overline \partial$ in equation \eqref{cr} at the zero section $O$ is defined for  $\psi \in W^{1,p}(p^{-1}A,R)$ by
\be \label{L}
L:W^{1,p}(p^{-1}A,R)\to L^p(\Omega^{01}(p^{-1}A)), \;\;
L(\psi)=d/dt|_{t=0} \bar{\partial}(O_{p^{-1}A} + t\psi). 
\ee
Let $z=(z_0, z_1) \in p^{-1}A$ with $p(z) = z_0$ be the components of a point in the bundle $p^{-1}A$. Then the map
$\tau (z_0, z_1)=(t_0, t_1)$ admits an inverse in the first component $t_0^{-1}(z_0, z_1)$ such that  
$z_0(z_0^{-1}, z_1)=z_0$, and we have 
\bean
 \tau {\rm graph}\; \varphi 
&=& \{ (t_0(z_0, \varphi), t_1(z_0, \varphi))  \; ; z_0 \in A \} \\
&=& {\rm graph}\; t_1 (t_0^{-1},\varphi )\\
&=:& {\rm  graph} \; t \;\varphi.
\eean
Here, last equality is the definition of $t$ as a map of sections.
Now we define the linearization $T$  of $t$ by $\nabla t \varphi=T \nabla \varphi$ with
\be \label{T}
(d  \tau)\; {\rm graph}\nabla \varphi={\rm graph}\; T \nabla \varphi.
\ee
Since $\tau$ is holomorphic, we have $J d \tau=d \tau J $  and $K T= T K$ therefore $\bar{\partial} t\varphi=T \bar{\partial}\varphi$.
\\
\begin{proposition} \label{p:lin}
The operators $L, T$ as in \eqref{L}, \eqref{T} have the following properties:
\begin{enumerate}\item[(a)] $L\psi=\nabla\psi+i
\nabla\psi\circ j + dJ_{21}(O_{p^{-1}A})\cdot \psi \circ j$,
\item[(b)] $L$ is Fredholm,
\item[(c)] L is T-equivariant: $ L\; T =dT\; L$. \end{enumerate} 
\end{proposition}
\begin{proof}
    Part (a) follows from (\ref{J0}) and (\ref{cr}). Note that $dJ_{21}(O_{p^{-1}A})$ is ${\mathbb{R}}$-linear rather than ${\mathbb{C}}$-linear. Part (b) holds since $L$ has the same symbol as the Cauchy-Riemann operator. Part (c) is seen as follows. 
\bean 
L S \psi &=&  d/dr|_{r=0} \;\bar{\partial}\; t \;\varphi_r \\
&=&  d/dr|_{r=0} \; (1-K) \nabla t \varphi_r \\
&=&  (1-d K) \; d T  \;\nabla \psi \\
&=&  d T\; L \psi.
\eean
\end{proof}

By the last Proposition, $L$ induces an ${\mathbb{R}}$-linear operator on  $ d \tau$-invariant sections, namely:
 $L_\tau:W^{1,p}_{\tau}(p^{-1}A,R))\to L^p_\tau(\Omega^{01}(p^{-1}A))$. For this operator we have the following.
\\
\begin{proposition} \label{tauequiv}
    For the $d \tau$-equivariant operator 
$L_\tau$ acting on $\tau$-invariant sections of $(p^{-1}A, R)$ we have: 
\begin{enumerate}
\item[(a)] $L_\tau$ is Fredholm,
\item[(b)] $index_{{\mathbb{R}}}\;  L_\tau=1$,
\item[(c)] $L_\tau$ is surjective. 
\end{enumerate} 
\end{proposition} 
\begin{proof}
 Claim (a) follows from Proposition \ref{p:lin} (b) and (c). 
    
To verify claim (b), we need to study the following elliptic complex $E$:
\[
0 \longrightarrow W^{1,p}(p^{-1}A,R))  \stackrel{L}\longrightarrow  
L^p(\Omega^{01}(p^{-1}A))  \longrightarrow 0.
\]
The homology groups of this complex are $H^0(E)= {\rm ker}\;L,\; H^1(E)= {\rm coker}\;L$.
By Proposition \ref{p:lin}, the group $\{\rm{id}, d \tau\}$ acts on  $E$. The quotient operator $L_\tau$ acts on invariant sections of $E$. Its index is given by 
\be \label{ind}
{\rm Index}\;\; L_\tau=\frac{1}{2}\sum_{g\in\{\rm{id}, d\hat{\tau}\}}{\rm Lef}\;(g,E).
\ee
Here,  ${\rm Lef}\;(g,E)$ indicates the Lefschetz number of the automorphism 
$g \in \{\rm{id},\tau\}$ : 

\[ {\rm Lef}\;(g,E)=
 \sum_{i=0,1}  (-1)^{i} {\rm trace_{{\mathbb{R}}} (g|_{H^i(E)})}.
\] 
 We therefore have Lef$(\rm{id}, E)=\mbox{index}\; L $, and by \cite{HLS} equation (2), we have ${\rm index}\; L=\mu + 2- \sigma -2g $. Since by Proposition \ref{tau} (a) the total Maslov index $ \mu$ vanishes, and $ \sigma $ denoting the number of  boundary components of the annulus, is two, and the genus $g$ of $A$ vanishes, we conclude that $ {\rm index}\; L=0$. 

Next recall that the fixed points of $ d \tau$ lie on the zero section   $A \times 0$  of the bundle $p^{-1}A $ are $\{-b/(2c),\infty\} \in A \times 0 \subset \mathbb{C}P^1$. By \cite{AS}, Theorem 3.1, 
the Lefschetz number is given by 
\be \label{lef}
{\rm Lef}(d \tau,E)=\sum_{a\; \in \;Fix(\;d\tau)}{\nu (a)},
\ee
where the $\nu$ involve alternating sums of the traces of $d\tau$ on $E$ and is defined by the equation:
\bean
[det_{\mathbb{R}} (\rm{id}-d\tau)|_{T_p A}] \;\nu(a)=
\mbox{tr}_{\mathbb{R}} \;d \tau |_{p^{-1} (a)} 
- \mbox{tr}_{\mathbb{R}}\; d\tau |_{\Omega^{01}( p^{-1}(a))}.
\eean
Since by equation \eqref{tau}, $d\tau (a) |_{T_a A} =-\rm{id}$ for the fixed points $a$ of $d\tau$, we have 
\[
det_{\mathbb{R}}(\rm{id}-d\tau)|T_q A=det_{\mathbb{R}} (-2\; \rm{id})=4,
\]
and compute
\bean
 \nu(a) &=& {\textstyle{\frac{1}{4}}}(\mbox{tr}\; 
d\tau|_{a}-
\mbox{tr} \;d\tau|_{\Omega^{01}(p^{-1}(a))}={\textstyle{\frac{1}{4}}}(2-(-2))=1. 
\eean
This implies by equation \eqref{lef} that $ {\rm Lef}(d \tau,E)=2 $. To justify these numbers, we consider the linear map $d\tau \in {\rm Aut}_{{\mathbb{R}}}(p^{-1}(a))$ acting on the fibre of \eqref{b} over the fixed point $a \in A$. Now we saw in the proof of Proposition 4 that this point corresponds to an eigenspace of $d\tau $ at the origin. Therefore, $d\tau$ acts as the identity on $p^{-1}(a) $ and has real trace two. 
Next we discuss $d\tau \in{\rm Aut}_{{\mathbb{R}}}(\Omega^{01}(p^{-1}(a)))$. We have $\Omega^{01}(p^{-1}(a))\cong\Omega^{01}(A) \otimes p^{-1}(a)$, and $d\tau$ acts as multiplication by -1 on $\Omega_{a}^{01}(A)$ as we see from the representation of $\tau$ being holomorphic in equation \eqref{matrixtau}, giving $(\mbox{tr}\;d\tau\; |_{\Omega^{01}(p^{-1}(a))}=-2$ for this term. This proves claim (b) by the formula for the index \eqref{ind}.

To prove claim (c), first consider the quotient bundle over the surface of disc-type
\[
(p^{-1}A_\tau, R_\tau) :=(p^{-1}A,R)/\{\rm{id},\tau \}\rightarrow (A_\tau, \partial A_\tau)=(A, \partial A)/\{\rm{id},\tau \}.
\]
Then $\Gamma(p^{-1}A_\tau, R_\tau)=\Gamma_\tau(p^{-1}A,R)$, which stands for the $\tau$-invariant sections of $p^{-1}A$ with boundary values in $R$. Next we double the quotient bundle as follows. The Riemann surface $ A_\tau$ is a pair $\{|A_\tau|, z_\alpha  \}$ of the underlying topological space and complex-valued holomorphic coordinates. Then $\overline{A_\tau}=\{|A_\tau|, \overline{z_\alpha}  \}$ is the mirror surface, and the double $A_\tau^d$ of $A_\tau$ is the topological space $  A_\tau^d= A_\tau  \cup   \overline{A_\tau}$, where $p \in \partial |A_\tau|$ is identified with $p \in \partial |\overline{A_\tau}|$, and is equipped with coordinates as follows. Let $z_\beta:A_\tau \supset U_p \to \{ \rm{Im}\: z_\beta > 0  \}$, $ \overline {z_\beta}: \overline {A_\tau} \supset \overline {U_p} \to \{ \rm{Im}\: z_\beta <  0  \}$ denote boundary charts of neighborhoods in $A_\tau$ of $p$ and in $\overline{A_\tau}$ of $p$, then  $z_\beta^d=\{ z_\beta, \overline {z_\beta} \}:U_p \cup \overline {U_p}\to{\mathbb{C}}$ is a holomorphic chart for the interior point $p \in A_\tau^d$.
In this fashion, $ A_\tau^d=\{|A_\tau| \cup (\overline{z_\alpha},z_\beta^d \}$ is a compact Riemann surface of genus zero. Note that the map $a\in A_\tau\to a \in \overline {A_\tau}$ extends to an anti-holomorphic involution of $ A_\tau^d$:
\[    
in:A_\tau^d \to A_\tau^d  ,\;\;\;\; in \circ  in =\rm{id}.
\]
The bundle $p^{-1}A_\tau$ may now be doubled to a bundle $p^{-1}A_\tau^d=p^{-1}A_\tau \; \cup \; \overline{p^{-1}A_\tau}$ over  $A_\tau^d$ by identifying the fibres $p^{-1}A_\tau(a)$, $ \overline{p^{-1}A_\tau(a)}$ over the boundary $a \in \partial A_\tau$ by reflection on the real line $R_\tau(a) \subset p^{-1}A_\tau(a)$.
The map $in$ lifts to an anti-holomorphic involution
\[   
\hat{ in}:p^{-1}A_\tau^d \to  p^{-1}A_\tau^d,
\;\;\;\; \hat{ in} \circ  \hat{ in} =\rm{id}.
\]
By \cite{MS12}, the first Chern class of the doubled bundle equals the total Maslov index of the relative bundle $(p^{-1}A_\tau , R_\tau)$: $c_1(p^{-1}A_\tau^d)=\mu(p^{-1}A_\tau , R_\tau)$. Now we define an operator 
\[ 
L^d_\tau:W^{1,p}(p^{-1}A_\tau^d)\to L^{1,p}\Omega^{01}(p^{-1}A_\tau^d),
\]
\[L^d_\tau|_{p^{-1}A_\tau}=L_\tau , \;\;
L^d_\tau|_{\overline{p^{-1}A_\tau}} = \hat{in} \circ L_\tau \circ in.
\]
Both parts are ${\mathbb{R}}$-linear differential operators and have the same symbol as the Cauchy-Riemann operator. Therefore $L^d_\tau$ is elliptic. 

We have $c_1(p^{-1}A_\tau^d)=\mu(p^{-1}A_\tau , R_\tau)=0$ by Proposition \ref{tau} (a). This verifies the inequality $c_1(p^{-1}A_\tau^d) \ge -1 $, which by  Theorem 1' and the Proposition on p158 of \cite{HLS} implies that  $L^d_\tau$ and $L_\tau$ are surjective over the surface $A^d_\tau$ of genus zero. This completes the proof of claim (c).
\end{proof} 

\noindent
{\bf Proof of Theorem \ref{t:1}:} 
By Proposition 8, the operator $L_\tau$ is surjective and, being of index one, has a 1-dimensional  kernel in $W^{1,p}$. By Theorem 4.18 of \cite{V} this implies that  the equation (\ref{cr0}) admits a one dimensional family of $\tau$-invariant solutions passing through the solution $ \varphi=0$,  which corresponds to the surface $(A, \partial A) \hookrightarrow (p^{-1}A,\Lambda_0^{\pm})$. 
Since $\Lambda_0^\pm$ is totally real in a neighborhood of $\partial A=S^\pm \subset  \Lambda_0^\pm$, the boundary problem \eqref{cr0} is strictly elliptic, and the boundary condition is $\Lambda_0^\pm \in C^{\frac{k}{2},\frac{\alpha}{2}}$, we conclude that the solutions of  (\ref{cr0}), which are of annulus-type, retain this regularity. We will now show that this family descends to a family in $\mathcal{D}(M, S)$ and thereby prove (b). For $r\in (0,1)$, let $D_r \subset D \subset \mathbb{C}$ denote the disc of radius $r$ in the unit disc and $A_r:= D \; \backslash \;D_r$. Then every domain of annulus-type is conformally equivalent to one  $A_r$, and the moduli space of $J$ - holomorphic curves of annulus-type in $(p^{-1}A,\Lambda_0^{\pm})$ is
\[
{\mathcal A}(p^{-1}A,\Lambda_0^{\pm})) : =\{
f {\rm\; mod\; Aut} A_r:f \in 
W^{1,p}(A_r,\partial A_r ; p^{-1}A,\Lambda_0^{\pm}), \bar\partial_{J}f=0,
r \in (0,1) \}.
\]
Let ${\mathcal A}_\tau(p^{-1}A,\Lambda_0^{\pm})$ be those elements of 
${\mathcal A(p^{-1}A,}\Lambda_0^{\pm})$ 
which invariant under the $J$-holomorphic involution $\tau$. They descend to $J$-holomorphic discs in ${\mathcal D}(M, S)$. This proves claim (a). By way of composing with  (\ref{e:co}) we obtain
\[
{\mathcal A}_\tau(p^{-1}A,\Lambda_0^{\pm})  
\stackrel{F_2 \circ {\rm Bl}_0^{-1}}
 \longrightarrow {\mathcal D}_q({M,S}),
\]
and its linearization at $A \times 0$
\[
{\mathcal A}_{d\tau}(T(p^{-1}A),T(\Lambda_0^{\pm})) \stackrel{d F_2 \cdot d{\rm Bl}_0^{-1} } 
\longrightarrow T_ q {\mathcal D}_q(M,S),
\]
where on the left hand side we have $J$-holomorphic annuli which are invariant under $d\tau$ and the right hand side is as in Proposition \eqref{p:4} . Therefore
\bean
T_ q {\mathcal D}(M,S) &=& T_0 (F_2 \circ {\rm Bl}_0^{-1} {\mathcal A}_\tau (p^{-1}A,\Lambda_0^{\pm}))\cr
 &=& dF_2  \cdot d{\rm Bl}_0^{-1} {\mathcal A}_{d\tau}(T(p^{-1}A),T(\Lambda_0^{\pm})) \cr
 &=& {\mathcal D}_0(\mathbb{C}^2, S_2).
\eean
The last space is given by the family \eqref{e:fa}, and this proves claim (b). The claim (a) follows from the remarks  on regularity above, and the fact that it is preserved under the transformations (\ref{e:co}) and $\tau$. \qed

\vspace{0.1in}
\noindent {\bf Proof of Theorem \ref{t:2}:} 
To prove (a) we consider the union of all embedded  discs of $\mathcal{D}_q(M,S)$. The claim (c) of Theorem \eqref{t:1} implies that near $q \in \mathcal{D}_q(M,S),$ they are disjoint as subsets of $(M,S).$ Claim (b) of  Theorem \ref{t:1} then implies that they form a real hypersurface with boundary in $(M,S).$ Denote this hypersurface by $N.$ Now claim (c) follows from claim (c) in Theorem \ref{t:1}.

\qed

\section{Holomorphic Discs near Complex points of higher order}\label{s:4}

Assume that for $F$ as in (\ref{e:1}), a point $ q \in  S$ is a complex point with local graphical Taylor expansion as follows:
\be \label{e:7}  
F(\zeta)=\left(\zeta, \zeta^{n-2}(a \zeta^2 + 2b \zeta \bar\zeta + c \bar \zeta^2) + O_{n+1}(\zeta, \bar\zeta)\right).
\ee
for some $ a, b, c \in \mathbb{C}$ $\zeta(q) =0$,   with $(b,c) \neq (0,0)$  and some  $n \ge 3$. This gives
\be
\partial F \wedge \bar\partial F = (2\zeta^{n-2}(b \zeta + c\bar{\zeta}) + O_n(\zeta, \bar \zeta))d\zeta \wedge d \bar\zeta.
\ee
This confirms that $q$ is a complex point of $S$ as per Definition \ref{d:4} and is degenerate for $n \ge 3$  and for $|b| - |c| > 0  \; ( < 0)$ has Maslov number $\nu(q) = n-3 \; (n-1)$. 

We proceed with the proof of Theorem \ref{t:3} that applies to the following surface.
\be \label{e:Shigher}
 S=\{ w=z^{n-2}(a z^2 + 2b z \bar z + c \bar{z}^2) + O_{n+1}(z,\bar z) \} \subset {\mathbb{C}}^2_{z,w},\quad  n \ge 3 \; {\rm and}  \;  |b|>|c|. 
\ee 
The first item thereof is the model family of discs.
\begin{proposition}\label{p:9}
Let $S$ be as in \eqref{e:7}, then we define
\begin{equation}\label{e:oscn}
  S_n:=\{ w=z^{n-2}(a z^2 + 2b z \bar z + c \bar{z}^2)\} \subset  \mathbb{C}^2_{z,w} ,  \quad
n \ge 3 \;\; {\rm and}  \;\;   |b|>|c|. 
\end{equation}
Consider the coordinates $(\zeta,\omega)$ as defined by: 
\[
(z,w)=\left(e^{i\delta}\zeta,\;\; \omega + (a e^{2i\delta}-{\textstyle{\frac{|c|}{|b|}}}b)\zeta^{n} \right),
\]
where $\delta=(1/2)(arg(b)-arg(c))$. Then the family ${\mathcal  D}_0({\mathbb{C}}^2,S_n)$ parametrized by $t>0$ of discs given by 
\begin{equation}
    (0,0) \in \Big\{ \left({\textstyle{\frac{|c|}{|b|}}} \zeta^{2} + 2 \zeta \bar \zeta +{\textstyle{\frac{|c|}{|b|} }}\bar \zeta^{2} < t,\;\; \omega=tb\zeta^{n-2}\right)\Big\}_{t>0} \subset \;{\mathbb{C}}^2_{\zeta ,\omega},
\end{equation}
holomorphic discs with boundary in  $S_n$ shrinking to the elliptic complex point $(0,0) \in  S_n$ as $t\to 0$.
\end{proposition}

Next we have an $n$-fold cover of the complex plane.
\be \label{e:covern}  F_n : \mathbb{C}^2_{z_1,w_1} \to \mathbb{C}^2_{z,w} , \;
F_n(z_1,w_1)=(z_1, z_1^{n-2}(a z_1^2 + 2b z_1 w_1 + c w_1^2) ).
\ee

\begin{proposition} \label{p:11n}
Let $q$ be a complex point of $S$  as in \eqref{e:7}, $\zeta$ a complex coordinate of $S$ with $\zeta(q) = 0$ and $F$ as in equation \eqref{e:7} and  $S_n$ be as in equation \eqref{e:oscn} and $F_n$ as in \eqref{e:covern}.  Then
\begin{enumerate}
\item[(a)] the map $F_n$ is a holomorphic $n$ -fold branched covering,and we have an associated holomorphic transformation:
\[
\tau:{\mathbb{C}}^2_{z_1,w_1}\to {\mathbb{C}}^2_{z_1,w_1},
\;\;
\;\;F_n \circ \tau =F_n , \;\; \tau \circ \tau = id,
\]
\item[(b)] there are two transversal totally real planes $ L^\pm_ q \subset
{\mathbb{C}}_{z_1,w_1}$ with $ L^+_ q \cup L^-_ q \subset F^{-1}_n ( S_n)$,
\item[(c)]  $|b| - |c| > 0 ( = 0, < 0)$ if and only if there is no (one, two) complex line(s) in ${\mathbb{C}}^2_{z_1,w_1}$ which intersect both planes $L^\pm_ q$ in real lines. In addition, the Keller-Maslov index of $q$ is $n-1$ ($n-3$) for $|b| - |c| >  0 (< 0).$
\end{enumerate}
\end{proposition}
\noindent
{\bf Proof :}
Claim (a) follows from solving for $\tau = (\tau_1,\tau_2)$  the equation 
$$(z_1, z_1^{n-2}(a z_1^2 + 2b z_1 w_1 + c w_1^2) ) = (\tau_1, \tau_1^{n-2}(a\tau_1^2 + 2b \tau_1 \tau_2 + c \tau_2^2)).$$ This gives a linear holomorphic map $\tau$ as in \eqref{matrixtau}. The proofs of claims (b) and (c) follow as in the section 2.\\ 
\qed \\\\
We proceed with the blow-via $F_n$ of the pair $(\mathbb{C}^2_{z_1,w_1}, L^\pm_q )$ and make reference to the diagram \eqref{diag}, where $F_2$ is replaces by $F_n$.
\begin{proposition} \label{p:blowup-n}
Assume that $ q, S, L_q^\pm, \tau$ be as in Proposition \ref{p:11n} and assume that $|b| - |c| > 0$ . Then there exists  $\epsilon > 0$ with
\begin{enumerate}
\item[(a)]  $F_n^{-1}( S \cap B_\epsilon) $ contains a union of two embedded totally real discs which we denote by $\Lambda^+ \cup \Lambda^- \subset {\mathbb{C}}^2_{z_1,w_1},$ and which intersect transversely at the origin and $T_{(0,0)}\Lambda^\pm=L^\pm_{ q}$,
\item[(b)] If $S\in C^{k,\alpha}$,  $k-(n-2) > 0$ , ($S \in C^\omega)$ then $\Lambda^\pm\in C^{(k-(n-2))/2, \alpha/2)}$ ($C^\omega$),
\item[(c)] Let 
${\rm Bl}_0(\mathbb{C}^2) \to \mathbb{C}^2_{z_1,w_1}$ 
be the blow-down map, then the proper transform $\Lambda_0^\pm :=  {\rm Bl}^{-1}_0(\Lambda^\pm) \hookrightarrow \rm{Bl}_0({\mathbb{C}}^2)$ of $\Lambda^\pm \subset{\mathbb{C}}^2_{z_1,w_1}$ are embedded, disjoint, and totally real M\"obius bands of  the same regularity as $\Lambda_0^\pm$ in (b),
\item[(d)] The holomorphic deck transformation $\tau$ of ${\mathbb{C}}^2_{z_1,w_1}$ extends to a holomorphic transformation also denoted $\tau : \rm{Bl}_0({\mathbb{C}}^2) \to \rm{Bl}_0({\mathbb{C}}^2).$ The fixed points of $\tau$ are on the exceptional fibre and correspond to the eigenspaces of the matrix  $(\ref{matrixtau})$. 
\end{enumerate}
\end{proposition}
\noindent
{\bf Proof}:
To prove claim (a) and (b) we need to solve the equation
\[
z_1^{n-2}(az_1^2 + 2bz_1w_1 + c w_1^2 ) = z_1^{n-2} (az_1^2 + 2bz_1\bar{z}_1 + c \bar{z}_1^2) + O_{n+1}(z_1,\bar{z}_1),
\]
where by assumption the term $O_{n+1}$ is $C^{k,\alpha}$ -regular. This implies that $z_1^{2-n}O_{n+1} = O_3(z_1,\bar z_1) \in C^{k-(n-2),\alpha}.$ We may now apply the Weierstrass Preparation Theorem and obtain a solution $w_1 = \bar z_1 + O_2(z_1,\bar z_1)$ of regularity $C^{(k-(n-2))/2,\alpha/2}.$ This is under the assumption that $k-(n-2) > 0.$ This proves (a) and (b) in analogy to section 2. Claims (c) and (d) follow as in in the proof of \eqref{p:blowup}.
\qed

\begin{proposition} \label{taun}
Let $(A, \partial A) \subset (p^{-1}A,\Lambda_0^\pm)$ be as in Proposition \ref{p:blowup-n} and $\tau$ as in Proposition\ref{p:11n}.  Then we have:
\begin{enumerate}
\item[(a)] The normal Maslov index $\mu(A, \Lambda_0^\pm)$ vanishes,
\item[(b)] $\tau(A,\partial^\pm A)=(A, \partial^\mp A)$, 
\item[(c)] The quotient $(A,\partial^\pm A)\to (A,\partial^\pm A) {\rm mod}{\{1,\tau}\}$ is diffeomorphic to a disc, where  the quotient map has two branching points.
\end{enumerate}
\end{proposition}
\noindent
{\bf Proof :} The proof of (a) is identical to that in Proposition \ref{tau}. In claims (b) and (c), the group generated by $\{ id, \tau\}$ is a subgroup of the deck group of $F_n$. Thereby, this group leaves invariant the the fibres of $F_n$ and, given that the zero section of $p^{-1}A$ contains two fixed points of $\tau$, the quotient is a smooth surface of vanishing genus, which proves claim (c).
\qed
\begin{proposition} \label{p:linn} Let $p^{-1}A,\Lambda_0^\pm, J, \psi, L,T$ be as Propositions \ref{taun} and \ref{p:blowup-n}
The ${\mathbb{R}}$ -linear differential operators L and T as in \eqref{L} have the following properties:
\begin{enumerate}\item[(a)]     $L\psi=\nabla\psi+i
\nabla\psi\circ j + dJ_{21}(O_{p^{-1}A})\cdot \psi \circ j$,
\item[(b)] $L$ is Fredholm,
\item[(c)] L is T-equivariant: $ L\; T =dT\; L$. \end{enumerate} 
\end{proposition}
\noindent
 {\bf Proof} : Claim (a) is proved identically as in Proposition \ref{p:lin} since the is does not depend on the regularity of $\Lambda_0^{\pm}.$ This implies claim (b) since the symbol of $L$ is that of the Cauchy-Riemann operator. Since  
 $\tau$ is $J$ -holomorphic,
 laim (c) finally again follows by linearization is in the proof of  Proposition \ref{p:lin}.

\qed \\
By the last Proposition, $L$ induces an ${\mathbb{R}}$-linear operator on the $ d \tau$-invariant sections, namely:
 $L_\tau:W^{1,p}_{\tau}(p^{-1}A,R))\to L^p_\tau(\Omega^{01}(p^{-1}A))$. For this operator we have the following.
\begin{proposition}
\label{Ltau}
For the $d \tau$-equivariant operator 
$L_\tau$ acting on $\tau$-invariant sections of $(p^{-1}A, R)$ we have: 
\begin{enumerate}
\item[(a)] $L_\tau$ is Fredholm,
\item[(b)] $index_{{\mathbb{R}}}\;  L_\tau=1$,
\item[(c)] $L_\tau$ is surjective. 
\end{enumerate} 
\end{proposition} 
 \noindent
 {\bf Proof :} Claim (a) follows as did claim (a) in
Proposition \ref{tauequiv}. The computation proving claim (b) of Proposition \ref{tauequiv} also applies here since it only uses $\mathbb{Z}_2$ -equivariance of the operator $L$ and the Maslov index of $A, \Lambda_0^{-1}$, as well as the Fredholm index of $L.$
\qed \\

\noindent
{\bf Proof of Theorem \ref{t:3}:}  We argue as in the proof of Theorem \ref{t:1}  that equation (\ref{cr0}) admits a one dimensional family of $ \tau$-invariant solutions passing through the solution $(A, \partial A) \hookrightarrow (p^{-1}A,\Lambda_0^{\pm})$. Since $\Lambda_0^\pm$ is totally real in a neighborhood of $\partial A=S^\pm \subset  \Lambda_0^\pm$, the boundary problem \eqref{cr0} is strictly elliptic. By Proposition \ref{p:blowup-n},  the boundary condition $\Lambda_0^\pm$ being of regularity class $C^{(k-(n-2))/2, \alpha/2)}$, the solutions of equation (\ref{cr0}), which are of annulus-type, retain this regularity.

For $r\in (0,1)$, let $D_r \subset D$ denote the disc of radius $r$ and $A_r:= D \; \backslash \;D_r$. Then every domain of annulus-type is conformally equivalent to one  $A_r$, and the moduli space of $J$-holomorphic curves of annulus-type in $(p^{-1}A,\Lambda_0^{\pm}))$ is
\[
{\mathcal A}(p^{-1}A,\Lambda_0^{\pm})=\{
f {\rm\; mod\; Aut} A_r:f \in 
W^{1,p}(A_r,\partial A_r ; p^{-1}A,\Lambda_0^{\pm}), \bar\partial_{J}f=0 \}.
\]
Let ${\mathcal A}_\tau(p^{-1}A,\Lambda_0^{\pm})$ be those elements of 
${\mathcal A(p^{-1}A,}\Lambda_0^{\pm})$ 
which are invariant under the $J$-holomorphic involution $\tau$. They descend to $J$-holomorphic discs in ${\mathcal D}(M, S)$ by way of composing with  the map (\ref{e:co}): 
\[
{\mathcal A}_\tau(p^{-1}A,\Lambda_0^{\pm})  
\stackrel{F_n \circ {\rm Bl}_0^{-1}}
 \longrightarrow {\mathcal D}_q({M,S}),
\]
and its linearization at $A \times 0$
\[
{\mathcal A}_{d\tau}(T(p^{-1}A),T(\Lambda_0^{\pm})) \stackrel{d F_n \cdot d{\rm Bl}_0^{-1} } 
\longrightarrow T_ q {\mathcal D}_q(M,S),
\]
where on the left hand side, we have $J$-holomorphic annuli which are invariant under $d\tau(0)$. Therefore
\bean
T_ q {\mathcal D}(M,S) &=& T_0 (F_n \circ {\rm Bl}_0^{-1} {\mathcal A}_\tau (p^{-1}A,\Lambda_0^{\pm}))\cr
 &=& dF_n  \cdot d{\rm Bl}_0^{-1} {\mathcal A}_{d\tau}(T(p^{-1}A),T(\Lambda_0^{\pm})) \cr
 &=& {\mathcal D}_0(\mathbb{C}^2, S_n).
\eean
The last space is given by the family \eqref{e:fa}, and this proves claim (b). Claim (a) follows from the regularity in Proposition \ref{p:blowup-n}, and the fact that it is preserved under the transformations (\ref{e:co}) and $\tau$.
\qed

\vspace{0.2in}
\noindent{\bf Statements and Declarations:}
 The authors have no financial or non-financial interests that are directly or indirectly related to this work. No data was collected during the course of this work. No AI was used in this work.

\end{document}